\documentclass[a4paper,11pt]{article}

\usepackage{titletoc}
\usepackage{fancyhdr}
\usepackage[utf8]{inputenc}
\usepackage{amssymb}
\usepackage{amsmath}
\usepackage{verbatim}
\usepackage{fullpage}
\usepackage[utf8]{inputenc}
\usepackage{setspace}
\usepackage{geometry}
\usepackage{setspace}
 \usepackage{amsmath,amscd}
\usepackage{a4wide}
\usepackage{amsthm}
\usepackage{color}
\usepackage{dsfont}
\numberwithin{equation}{section}
\usepackage{ifpdf}
\ifpdf
    \usepackage[pdftex]{graphicx}   % to include graphics
    \usepackage[pdftex,     % sets up hyperref to use pdftex driver
            plainpages=false,   % allows page i and 1 to exist in the same document
            breaklinks=true,    % link texts can be broken at the end of line
            colorlinks=false,
            pdftitle=My Document
            pdfauthor=My Good Self
           ]{hyperref}
    \usepackage{thumbpdf}
\else
    \usepackage{graphicx}       % to include graphics
    \usepackage{hyperref} {colorlinks=false}      % to simplify the use of \href
\fi

\newcommand{\R}{\mathbb{R}}

\newcommand{\p}{\partial}

\newtheorem{Theorem}{Theorem}[section]

\newtheorem{Definition}{Definition}[section]

\newtheorem{Proposition}{Proposition}[section]

\begin{document}

 \title {The Mean Field Kinetic Equation for Interacting Particle Systems with Non-Lipschitz Force}

\author { Qitao Yin\footnotemark[1]\, \quad Li Chen\footnotemark[1]  \, \quad Simone G\"ottlich\footnotemark[1] }

\footnotetext[1]{University of Mannheim, Department of Mathematics,
68131 Mannheim, Germany. E-mail: qyin@mail.uni-mannheim.de; \{chen, goettlich\}@math.uni-mannheim.de.}

 \maketitle

\begin{abstract}
In this paper, we prove the global existence of the weak solution to the mean field kinetic equation derived from the $N$-particle Newtonian system. For $L^1\cap L^\infty$ initial data, the solvability of the mean field kinetic equation can be obtained by using uniform estimates and compactness arguments while the difficulties arising from the non-local non-linear interaction are tackled appropriately using the Aubin-Lions compact embedding theorem.
%The results in this paper have been used in \cite{CGY16} to achieve the rigorous derivation of the mean field kinetic equation for the pedestrian flow model.
%% under suitable assumptions
\end{abstract}

%%%%%%%%%%%%%%%%%%%%%%%%%%%%%%%%%%%%%%%%%%%%%%%%%%%%%%%%%%%%%%%%
\section{Introduction}
%%%%%%%%%%%%%%%%%%%%%%%%%%%%%%%%%%%%%%%%%%%%%%%%%%%%%%%%%%%%%%%%

In this paper, we investigate a two-dimensional kinetic mean field equation for the mass distribution $f(t,x,v)$ 
with position $x \in \R^2$ and velocity $v \in \R^2$ given by
\begin{eqnarray}\label{meanfieldequation}
\label{VE}
\p_tf + v \cdot \nabla_xf +\nabla_v \cdot \left[(F*f)f\right]+\nabla_v \cdot (G f)=0. %\quad x, v \in \R^2.
\end{eqnarray}

%Social phenomena and in particular pedestrian flow models have been experienced a great increase in interest over the last few years.
%Different models have been developed and investigated from a numerical and theoretical point of view,
%see for example \cite{Bellomo2011,Cristiani2011,Naldi2010} for a general overview.
%The modeling of pedestrian crowds is inspired by fluid dynamics
%and the consideration of humans interactions, so-called social fields \cite{Lewin51} or social forces \cite{HM95}.
%Comparisons with empirical data \cite{HM95,Zhang2012} allow then for the observation of new spatiotemporal patterns (e.g. line formation, freezing by heating)
%in pedestrian motions. Further applications of such behavioral models include group dynamics \cite{Bellomo2012},
%minimal travel times \cite{DiFrancesco2011,Hughes2001} or evacuation scenarios \cite{Piccoli2009,Twarogowska2014}.
Equation \eqref{meanfieldequation} is motivated by several applications such as crowd dynamics \cite{HM95,Lewin51}
or material flow \cite{GKT15} and has been investigated from a numerical and theoretical point of view,
see for example \cite{Bellomo2011,Cristiani2011,Naldi2010} for a general overview. 
Further extensions might be behavioral models including group dynamics \cite{Bellomo2012},
minimal travel times \cite{DiFrancesco2011,Hughes2001} or evacuation scenarios \cite{Piccoli2009,Twarogowska2014}. 
Model hierarchies for pedestrian and material flow applications have been introduced in \cite{Degond2013,EGKT14,GHSSV14,GKT15}, 
where macroscopic equations are formally derived from a microscopic Newtonian system. 
Depending on the closure assumption, different non-local continuum models can occur, cf.~\cite{Colombo2008}.
However, from an analytical point of view, there are still open problems that need to be thoroughly investigated
as for instance the detailed derivation from the $N$-particle (pedestrian) Newtonian system to its mean field limit or Vlasov equation, see \cite{CGY16}.
Instead of the formal derivation with the help of the BBGKY hierarchy \cite{EGKT14, Sophn12}, the kinetic description
has been rigorously derived by a probabilistic method \cite{BP15,BH77,HJ07,HJ11,Philip2007,Sznitman91}.

In this paper, we now aim to prove the global existence of the weak solution
to the mean field kinetic equation \eqref{meanfieldequation}.
In the latter equation,  $F(x,v)$ denotes the total interaction force and has the similar structure as $\displaystyle \frac{x}{|x|}$, i.e.,
$$
F(x,v)=\nabla_x V(|x|,v)=\partial_rV(r,v)\frac{x}{|x|},
$$
where $V(|x|,v)$ is some (regular) potential.
More precisely, $F(x,v)$ can be a composition of the interaction force $F_{int}(x)$ and the dissipative force
$ F_{diss}(x,v)$, i.e.,
\begin{equation}
F(x,v)=(F_{int}(x)+F_{diss}(x,v))\mathcal{H}(x,v) \label{F}
\end{equation}
and $\mathcal{H}(x,v):=\mathcal{H}_{2R}(|x|) \cdot \widetilde{\mathcal{H}}_{2\widetilde{R}}(|v|)$,
where $\mathcal{H}_{2R}(|x|)$ and $ \widetilde{\mathcal{H}}_{2\widetilde{R}}(|v|)$
are smooth functions with compact support such that
$$
\mathcal{H}_{2R}(|x|)=\begin{cases}
0, & |x| >2R,\\
1, & |x| < R,
\end{cases}
\quad \hbox{and} \quad
 \widetilde{\mathcal{H}}_{2\widetilde{R}}(|v|)=\begin{cases}
 0, & |v| >2\widetilde{R},\\
 1, & |v| < \widetilde{R}.
\end{cases}
$$
  In order to cover a realistic behavior of moving crowds, the functions $\mathcal{H}_{2R}(|x|)$ and
	$\widetilde{\mathcal{H}}_{2\widetilde{R}}(|v|)$
	are used to express that the interaction force and the velocity of agents
  are of finite range. So the total force is considered on a bounded domain.

The other term $G(x,v)$ in equation \eqref{meanfieldequation} represents the desired velocity
and the direction acceleration and can be further written as
\begin{equation}\label{Gg}
G(x,v)=g(x)-v,
\end{equation}
where  $\displaystyle \|g\|_{L^{\infty}}$ is bounded by some constant.

Apparently, the proposed model equation \eqref{meanfieldequation} involves a singularity comparable to
the Coulomb potential in 2-$d$, resulting from the total interaction force.
That means this singularity, or in other words the non-local term, needs extra care in the final limiting process.
For more information about the Coulomb potential and the Vlasov-Poisson system, we refer to \cite{Pfaffelmoser92,Rein98, Schaeffer91}.

We now briefly explain our approach to obtain the existence of the weak solution. 
First, we consider an approximate problem (kinetic equation with cut-off) and show that the approximate problem has a weak solution,
where the mean field characteristic flow is of great importance.
%, proven by using Banach fixed point theorem
Unlike the 3-$d$ Vlasov-Poisson equation \cite{Dobrushin79, LP91}, the non-local operator in \eqref{VE} cannot be decoupled into
an elliptic equation. Hence, the Calder\'{o}n-Zygmund continuity theorem \cite{Golse13} for second order elliptic equations is not applicable in this case and
we have to find an alternative way to fix the desired compactness arguments.
The idea is to use the Aubin-Lions lemma \cite{CJL14,Simon86} and to argue that
due to that compact embedding theorem, we are able to pass the limit especially in the non-local term. % and complete the whole proof.
We also remark that the result obtained in the present paper plays a crucial role in the proof of the rigorous derivation of the mean field equation
in \cite{CGY16}.

This article is organized as follows: In Section~\ref{sec:mean}, we state our main result and further introduce
some notations and preliminary work to
show that the characteristic flow associated with the cut-off mean field equation admits a unique solution.
We also prove the existence and uniqueness of the weak solution to the cut-off mean field equation.
Section~\ref{sec:compact} is concerned with the compactness arguments that are needed to pass the limit and to obtain the desired weak formulation of the non-cut-off kinetic equation. However, the corresponding uniqueness can no longer be kept during the limiting procedure. Finally, we summarize our results.

\section{Mean Field Equation with Cut-off}  \label{sec:mean}
We start with the definition of a weak solution to the mean field equation \eqref{meanfieldequation}.
\begin{Definition}
	  Let $f_0(x,v) \in L^1(\R^2 \times \R^2) \cap L^{\infty}(\R^2 \times \R^2)$. A function $f=f(t,x,v)$ is said to be a weak solution to the kinetic mean field equation \eqref{VE} with initial data $f_0$, if there holds
\begin{eqnarray}
 &&\iint_{\R^2 \times \R^2} f(t,x,v) \varphi(x,v)
  \,dxdv = \iint_{\R^2 \times \R^2}f_0(x,v) \varphi(x,v)
  \,dxdv \nonumber \\
  && \hspace{16mm} +\int_0^t \iint_{\R^2 \times \R^2}vf(s,x,v) \cdot\nabla_x \varphi(x,v)
  \,dxdvds  \nonumber \\
  && \hspace{16mm}+\int_0^t \iint_{\R^2 \times \R^2} \left(F(x,v)*f(s,x,v)\right) f(s,x,v)\cdot\nabla_v \varphi(x,v)\,dxdvds  \nonumber \\
    &&\hspace{16mm}+ \int_0^t \iint_{\R^2 \times \R^2}G(x,v) f(s,x,v)\cdot \nabla_v \varphi(x,v) \,dxdvds
  \end{eqnarray}
for all $\varphi(x,v) \in C^{\infty}_0(\R^2 \times \R^2)$ and $t \in \R_+$.
\end{Definition}

Next, we present the main theorem of this paper. In the following, $G(x,v)$ is given by \eqref{Gg} while $F(x,v)$ is defined by \eqref{F}.

\begin{Theorem}\label{thm:1}
 For $F(x,v)=\nabla_x V(|x|,v)=\partial_rV(r,v)\frac{x}{|x|}$ and $G(x,v)=g(x)-v$, assume that $\partial_rV(r,v),\nabla_v\partial_rV(r,v)\in L^\infty(\R^2 \times \R^2)$ and $g\in L^\infty(\R^2 \times \R^2)$. Let $f_0(x,v)$ be a nonnegative function in $ L^1(\R^2 \times \R^2) \cap L^{\infty}(\R^2 \times \R^2)$, $|x|^2f_0(x,v) \in L^1(\R^2 \times \R^2)$, and
$$
\iint_{\R^{2}\times \R^{2}}\frac{1}{2}|v|^{2}f_{0}(x,v)\,dxdv=: \mathcal{E}_0<\infty.
$$
%Furthermore, $F$ is given by \eqref{F} and $G$ is Lipschitz continuous.
Then, there exists a weak solution $f \in L^{\infty}(\R_{+};L^{1}(\R^{2} \times \R^{2}))$ to the mean field equation \eqref{VE} with initial data $f_0$. Moreover this solution satisfies
\begin{equation} \label{maximum}
  0 \leq f(t,x,v) \leq \Vert f_0 \Vert_{L^{\infty}(\R^{2}\times \R^{2})}e^{Ct}, \quad \hbox{for a.e.}  \,(x,v) \in \R^{2} \times \R^{2}, \,  t \geq 0
\end{equation}
together with the mass conservation
\begin{equation} \label{mass}
    \iint_{\R^{2}\times \R^{2}}f(t,x,v) \,dxdv = \iint_{\R^{2}\times \R^{2}}f_{0}(x, v) \,dxdv=: \mathcal{M}_0
\end{equation}
and the kinetic energy  bound
\begin{equation} \label{kinetic}
   \mathcal{E}(t):=\iint_{\R^{2}\times \R^{2}}\frac{1}{2}|v|^{2}f(t, x, v)\,dxdv \leq C, \quad \forall \, t \geq 0,
\end{equation}
where the constant $C$ is independent of $t$. %The initial condition is verified in the sense of distributions, i.e., for all $\varphi \in C_{0}^{\infty}(\R^{2}\times \R^{2})$ ,  the function
%$$
%t \mapsto \iint_{\R^{2}\times \R^{2}}f(t, x, v)\varphi(x, v)\,dxdv
%$$
% is continuous on $\R_{+}$  and satisfies
%$$
%\lim_{t\to 0^+}\iint_{\R^{2}\times \R^{2}}f(t, x, v)\varphi(x, v)\,dxdv=\iint_{\R^{2}\times \R^{2}}f_0(x, v)\varphi(x, v)\,dxdv.
%$$
\end{Theorem}

Under the assumptions above, the interaction force is bounded but not Lipschitz continuous in $x$. We need to use the standard cut-off to overcome this difficulty. Another difficulty in this context is that the interaction force $F(x,v)$ not only depends on the position $x$ but also on the velocity $v$.
This leads to a totally different structure compared to the Vlasov-Poisson equation, where the $W^{2,p}$ theory for Poisson equations is generally used.
The proof of Theorem \ref{thm:1} is therefore not as straightforward and intuitive as expected and therefore needs to be dedicately handled step by step within the next sections. On the other hand, the self-generating force (or desired velocity and direction acceleration) $G(x,v)$ is not Lipschitz continuous, which requires an additional work of mollification.

%The proof Theorem \ref{thm:1} is quite involved so we arrange it in the following sections.  \\\\

%%%%%%%%%%%%%%%%%%%%%%%%%%%%%%%%%%%%%%%%%%%%%%%%%%%%%%%%%%%

%%%%%%%%%%%%%%%%%%%%%%%%%%%%%%%%%%%%%%%%%%%%%%%%%%%%%%%%%%%

We briefly recall essential assumptions and properties, cf. \cite{CGY16},
that are necessary for the existence proof.

%%%%%%%%%%%%%%%%%%%%%%%%%%%%%%%%%%%%%%%%%%%%%%%%%%%%%%%%%%%
\subsection{Notations and Preliminary Work} \label{sec:prelim}
%%%%%%%%%%%%%%%%%%%%%%%%%%%%%%%%%%%%%%%%%%%%%%%%%%%%%%%%%%%

We consider the flow with cut-off of order $N^{-\theta}$ with arbitrary positive $\displaystyle  \theta $, i.e.,
\begin{equation}\label{FN}
F^N(x,v)=\begin{cases}
\displaystyle V'(|x|,v)\frac{x}{|x|} \mathcal{H}(x,v), & |x| \ge N^{-\theta}, \\[5mm]
\displaystyle N^{\theta}V'(|x|,v)x \mathcal{H}(x,v), & |x| < N^{-\theta}. \end{cases}
\end{equation}
Then, the mean field cut-off equation becomes
 \begin{eqnarray} \label{vlasovN}
      \p_t f^N+v \cdot \nabla_x f^N+\nabla_v \cdot [(F^N*f^N) f^N]+\nabla_v \cdot (G^N
      f^N)=0,
    \end{eqnarray}
where we also take the cut-off of $G(x,v)$ into consideration, i.e.,
\begin{eqnarray*}
	G^N(x,v)=j_{\frac1N}*g(x)-v
\end{eqnarray*} %\label{GN}
with $\displaystyle j_{\frac1N}(x)$ being the standard mollifier.
%\begin{Definition}
% Let $(\overline{X}^N_t, \overline{V}^N_t)$ be the trajectory on $\R^{4N}$ which evolves according to
%    the Vlasov equation
%    \begin{eqnarray} \label{vlasovN}
%      \p_t f^N+v \cdot \nabla_x f^N+\nabla_v \cdot [(F^N*f^N) f^N]+\nabla_v \cdot (G
%      f^N)=0,
%    \end{eqnarray}
%
%    i.e.,
%    \begin{align}
%\label{VF} \begin{cases}
% \displaystyle \frac{d}{dt}\overline{X}^N_t=\overline{V}^N_{t}, \\ \vspace{0.005cm}\\
% \displaystyle{\frac{d}{dt}}\overline{V}^N_t=\overline{\Psi}^N(\overline{X}^N_t, \overline{V}^N_t)+\Gamma(\overline{X}^N_t, \overline{V}^N_t),
% \end{cases}
% \end{align}
%    where $\displaystyle \big(\overline{\Psi}^N(\overline{X}^N_t, \overline{V}^N_t)\big)_i=\iint F^N(\overline{x}^N_i-y, \overline{v}^N_i-w)f^N(t, y,w)\,dydw$ and
%    $\big(\Gamma(\overline{X}^N_t,
%    \overline{V}^N_t)\big)_i=G(\overline{x}^N_i,\overline{v}^N_i)$ represent the total
%    interaction force and the desired velocity and direction
%    acceleration, respectively.
%
%\end{Definition}

%\noindent  If $N$ is removed from the superscript, then $(X_t, V_t)$ and $(\overline{X}_t, \overline{V}_t)$
%  denote the particle configurations driven by the force without cut-off.
%  Analogically, if $t$ is removed from the subscript, $(X, V)$ and $(\overline{X}, \overline{V})$
%  represent the stochastic initial data, which are independent and identically distributed.
% Note that we always consider the same initial data for both systems, that means $(X, V) = (\overline{X},
%  \overline{V})$.

We also point out several properties for the interaction force $F^N(x,v)$ and the acceleration $G^N(x,v)$, namely
%and the acceleration $G(x,v)$. %All the properties can be checked by direct computations.
  \begin{itemize}
    \item[(a)] $F^N(x,v)$ is bounded, i.e., $|F^N(x,v)| \le C$.
    \item[(b)] $F^N(x,v)$ satisfies
    $$|F^N(x,v)-F^N(y,v)| \le q^N(x,v) |x-y|,$$
   where $q^N$ has compact support in $B_{2R} \times B_{2\widetilde{R}}$ with
            $$
            q^N(x,v):=
              \begin{cases} \displaystyle
                C \cdot \frac{1}{|x|}+C, & |x|\ge N^{-\theta}, \\[3mm]
                C \cdot N^{\theta}, & |x|< N^{-\theta}.
              \end{cases}
              $$
    \item[(c)] $\nabla_v F^N(x,v)$ is uniformly bounded in $N$.
    \item[(d)] $|G^N(x,v)-G^N(y,v)|\leq C \cdot N \cdot |x-y|$.
   % \item[(d)] $G(x,v)$ is bounded, i.e., $|G(x,v)| \le C$.
  \end{itemize}
	
  Here, we use $C$ as a universal constant that might depend on all the given constants $k_n, R, \widetilde{R}, \gamma_n, \gamma_{t}.$

Furthermore, if there is a singularity in the velocity $v$ in the interaction
  potential similar to property (b),
%i.e.,
%  $$|F^N(x,v)-F^N(x,w)| \le \widetilde{q}^N(x,v) |v-w|,$$
%  where $\widetilde{q}^N(x,v)$ has compact support in $B_{2R} \times B_{2\widetilde{R}}$ with
% $$
%            \widetilde{q}^N(x,v):=
%              \begin{cases} \displaystyle
%                C \cdot \frac{1}{|v|}+C, & |v|\ge N^{-\theta}, \\ \vspace{0.005cm}\\
%                C \cdot N^{\theta}, & |v|< N^{-\theta},
%             \end{cases}
%            $$
it can be treated by using the same method as above and the results also
apply.

%%%%%%%%%%%%%%%%%%%%%%%%%%%%%%%%%%%%%%%%%%%%%%%%%%%%%%%%%%%%%%%%%%%%%%%%%
\subsection{Mean Field Characteristic Flow with Cut-off}
%\label{sec:mean}
%%%%%%%%%%%%%%%%%%%%%%%%%%%%%%%%%%%%%%%%%%%%%%%%%%%%%%%%%%%%%%%%%%%%%%%%%

Before we start to prove the existence of the unique weak solution to the equation \eqref{vlasovN}, we need first the following definition.
\begin{Definition}
	Let  $(X_1, \Sigma_1)$ and $(X_2, \Sigma_2)$ be measurable spaces (meaning that $\Sigma_1$ and $\Sigma_2$ are $\sigma$-algebras of the subsets of $X_1$ and $X_2$, respectively). Let $T: X_1 \rightarrow X_2$  be a $(\Sigma_1,\Sigma_2)$-measurable map and $\mu$ be a positive measure on $(X_1, \Sigma_1)$. Then, the formula
	$$
	\nu (B) := \mu (T^{-1}(B)), \quad \forall \,B \in  \Sigma_2
	$$
	defines a positive measure on $(X_2, \Sigma_2)$, denoted by
	$$
	\nu =: T\#\mu,
	$$
	and is referred to as the push-forward of the measure $\mu$ under the map $T$.
\end{Definition}

The definition is often used when it comes to solving mean field characteristic flow. For more detailed information, we refer to \cite{Golse13}.
Due to the property of the transport equation, we know that solving the equation \eqref{vlasovN} is equivalent to investigating the corresponding characteristic system, i.e.,
	 \begin{eqnarray}
 \begin{cases}  \label{meanfieldZ}
 \displaystyle \frac{d}{dt}Z(t, z_0, \mu_0)=\int_{\R^4}K\left(Z(t, z_0),z'\right) \mu(t,dz'), \\
% \displaystyle \mu(t, \cdot)=Z(t, \cdot, \mu_0)\#\mu_0,  \\ %\vspace{0.1cm}
 Z(0, z_0, \mu_0)=z_0,
 \end{cases}
\end{eqnarray}
where
$$
K^N(z,z')=K^N(x,v,x',v'):=\left(v, F^N(x-x',v-v')+G^N(x,v) \right)
$$
and $ \mu(t, \cdot)$ is the push-forward of the measure $\mu_0$. Here, for the sake of convenience, we use $z=(x,v)$ and $Z$ as the four-dimensional vector.

We denote $\mathcal{P}(\R^4) $ as the set of Borel probability measures on $\R^4$ and $\mathcal{P}_1(\R^4)$ is defined by %as the set of Borel probability measures on $\R^4$ with finite moment (with respect to the second variable) of order one, i.e.,
	$$
	\mathcal{P}_1(\R^4):=\Big\{  \mu \in \mathcal{P}(\R^4) \, \Big| \, \int_{\R^4} |v| \mu(dx,dv) < \infty \Big \}.
	$$	
	
	\begin{Proposition} \label{prop3.1}
	Assume that the interaction kernel $K(z,z') \in C(\R^4 \times \R^4; \R^4)$ is Lipschitz continuous in $z$, uniformly in $z'$ (and conversely), i.e., there exists a constant $L>0$ such that
\begin{eqnarray*}
	&&\sup_{z' \in \R^4} |K(z_1,z')-K(z_2,z')| \le L|z_1-z_2|,\\
	&&\sup_{z \in \R^4} |K(z,z_1)-K(z,z_2)| \le L|z_1-z_2|.
\end{eqnarray*}
	For any given $z_0=(x_0,v_0) \in \R^2\times \R^2$ and Borel probability measure $\mu_0 \in \mathcal{P}_1(\R^4)$, there exists a unique $C^1$-solution, denoted by
	$$
	\R_+ \ni t \mapsto Z(t, z_0, \mu_0) \in \R^4,
	$$
to the problem
	 \begin{eqnarray}
 \begin{cases}  \label{meanfieldZ}
 \displaystyle \frac{d}{dt}Z(t, z_0, \mu_0)=\int_{\R^4}K\left(Z(t, z_0),z'\right) \mu(t,dz'), \\
% \displaystyle \mu(t, \cdot)=Z(t, \cdot, \mu_0)\#\mu_0,  \\ %\vspace{0.1cm}
 Z(0, z_0, \mu_0)=z_0,
 \end{cases}
\end{eqnarray}
	where $ \mu(t, \cdot)$ is the push-forward of the measure $\mu_0$, i.e., $\mu(t, \cdot)=Z(t, \cdot, \mu_0)\#\mu_0$.
\end{Proposition}

This proposition is typically obtained via the standard argument using Banach Fixed-Point Theorem, see \cite{Golse13}.

With Proposition \ref{prop3.1}, we are now able to prove that there exists a unique weak solution to the Vlasov equation with cut-off \eqref{vlasovN}.
%$F^N$ shall now be given by \eqref{FN}.

\begin{Theorem}\label{thmexistN} Let $F$ and $G$ satisfy the same assumptions as in Theorem \ref{thm:1} and $f^N_0$ be a nonnegative compactly supported function in $ L^1(\R^2 \times \R^2) \cap L^{\infty}(\R^2 \times \R^2)$ %which has support in $B_N$ , where $B_N$ is a ball of radius $N$ centered at origin,
	   satisfying
$$
\|f^N_0\|_{L^1(\R^{2}\times \R^{2})}= \mathcal M_0 \quad\mbox{ and }\quad f^N_0(x,v) \leq \|f_0 \|_{L^{\infty}(\R^{2}\times \R^{2})},
$$
%$$
 %\iint_{\R^{2}\times \R^{2}}|v|f_0(x,v)\,dxdv<\infty
%$$
$$
\iint_{\R^{2}\times \R^{2}}\frac{1}{2}|v|^{2}f^N_{0}(x,v)\,dxdv \le  \mathcal{E}_0<\infty,
$$
 and
$$
\iint_{\R^{2}\times \R^{2}}\frac{1}{2}|x|^{2}f^N_{0}(x,v)\,dxdv \le  \mathcal{M}_2<\infty.
$$
%Furthermore, $F^N$ is given by \eqref{FN} and $G$ is Lipschitz continuous.
Then, there exists a unique weak solution $f^N \in C^1(\R_{+};L^{1}(\R^{2} \times \R^{2}))$ to the mean field cut-off equation \eqref{vlasovN} with initial data $f^N_0$, i.e., $f^N(t,x,v)$ satisfies
\begin{eqnarray}
 &&\iint_{\R^2 \times \R^2} \p_t f^N(t,x,v) \varphi(x,v)
  \,dxdv =  \iint_{\R^2 \times \R^2}vf^N(t,x,v) \cdot\nabla_x \varphi(x,v)
  \,dxdv  \nonumber \\
  && \hspace{16mm}+ \iint_{\R^2 \times \R^2} \left(F^N(x,v)*f^N(t,x,v)\right) f^N(s,x,v)\cdot\nabla_v \varphi(x,v)\,dxdv  \nonumber \\
    &&\hspace{16mm}+ \iint_{\R^2 \times \R^2}G^N(x,v) f^N(t,x,v)\cdot \nabla_v \varphi(x,v) \,dxdv
  \end{eqnarray}
for all $\varphi(x,v) \in C^{\infty}_0(\R^2 \times \R^2).$
Moreover this solution satisfies
$$
\lim_{t \rightarrow 0}f^N(t,x,v)=f_0^N(x,v),\quad \hbox{for a.e.}  \,(x,v) \in \R^{2} \times \R^{2},
$$
\begin{equation} \label{maximumN}
  0 \leq f^N(t,x,v) \leq \|f^N_0 \|_{L^{\infty}(\R^{2}\times \R^{2})}e^{Ct}, \quad \hbox{for a.e.}  \,(x,v) \in \R^{2} \times \R^{2}, \,  t \geq 0
\end{equation}
together with the mass conservation
\begin{equation} \label{massN}
    \iint_{\R^{2}\times \R^{2}}f^N(t,x,v) \,dxdv=\iint_{\R^{2}\times \R^{2}}f^N_{0}(x, v) \,dxdv =:  \mathcal{M}_0,
\end{equation}
the kinetic energy  bound
\begin{equation} \label{kineticN}
  \iint_{\R^{2}\times \R^{2}}\frac{1}{2}|v|^{2}f^N(t, x, v)\,dxdv \leq C, \quad \forall \, t \geq 0,
\end{equation}
and the bound of second moment
\begin{equation} \label{2momentN}
  \iint_{\R^{2}\times \R^{2}}\frac{1}{2}|x|^{2}f^N(t, x, v)\,dxdv \leq \mathcal{M}_2e^{Ct}, \quad \forall \, t \geq 0,
\end{equation}
where the constant $C$ is independent of $N$ and $t$.
 \end{Theorem}

\begin{proof} Without loss of generality, we assume that $\mathcal M_0=1$.
	If we choose the interaction kernel $K$ as
	$$
	K^N(z,z')=K^N(x,v,x',v'):=\left(v, F^N(x-x',v-v')+G^N(x,v) \right),
	$$
	the mean field cut-off equation \eqref{vlasovN} can be put into the form
	$$
\partial_{t}f^N(t, z)+\hbox{div} _{z}\Big(f^N(t, z)\iint_{\R^{2}\times \R^{2}}K^N(z, z')f^N(t, z')dz' \Big)=0.
$$
Notice that the non-linear non-local dynamical system that appears in Proposition \ref{prop3.1} is exactly the equation of characteristics for the mean field kinetic equation with cut-off \eqref{vlasovN}, which we refer to as the mean field characteristic flow (with cut-off).
The existence and uniqueness of the solution to \eqref{vlasovN} are therefore achieved as a direct result of the construction of the mean field characteristic flow. By Proposition \ref{prop3.1}, there exists a unique map
$$
\R_{+}\times \R^{4}\times \mathcal{P}_{1}(\R^{4})\ni(t, z_0, \mu_0)\mapsto Z^N(t, z_0, \mu_0)\in \R^{4}
$$
such that $t \mapsto Z^N(t, z_0, \mu_0)$ is the integral curve of the vector field
$$
z \mapsto \iint_{\R^2 \times \R^2} K^N(z, z')\mu^N(t, dz') %=:((t))(z)
$$
passing through $z_0$ at time $t=0$, where $\mu^N(t) :=Z^N(t, \cdot, \mu_0)\#\mu_0.$
For the given initial data $f^N_0$, letting $d\mu_0=f^N_0dz$ results in
$$
f^N(t, z):=f^N_0\left(Z^N(t, \cdot)^{-1}(z)\right)J(0,t,z),\quad \forall \, t\geq 0,
$$
where $J(0,t,z)$ is the Jacobian, i.e.,
$$
J(0,t,z)=\exp{\left( \int_t^0 \hbox{div}_v \left(F^N\ast f^N(s,Z^N(s,z)) +G^N(Z^N(s,z))\right)\,ds \right)}.
$$
Then we have
\begin{eqnarray*}
	|f^N(t, z)| &\leq & |f^N_0\left(Z^N(t, \cdot)^{-1}(z)\right)J(0,t,z)|\\
					&\leq & \|f^N_0 \|_{L^{\infty}(\R^{2}\times \R^{2})}\exp{\left( \int_0^t \| \nabla_v F^N\ast f^N \|_{L^{\infty}(\R^{2}\times \R^{2})} \,ds +Ct\right)} \\
					&\leq &  \|f^N_0 \|_{L^{\infty}(\R^{2}\times \R^{2})}\exp{\left(\int_0^t \| \nabla_v F^N\|_{L^{\infty}(\R^{2}\times \R^{2})} \|f^N \|_{L^{1}(\R^{2}\times \R^{2})}  \,ds +Ct\right)} \\
					&\leq &  \|f^N_0 \|_{L^{\infty}(\R^{2}\times \R^{2})} e^{Ct},
\end{eqnarray*}
where we have used the property of the acceleration $G^N(x,v)$, i.e., $\displaystyle G^N(x,v)=j_{\frac1N}*g(x)-v$, where $\displaystyle  j_{\frac1N}*g(x)$ is a $L^\infty$-function.
From the equation, \eqref{massN}
are straightforward. Property \eqref{kineticN} is left to be proven. For the kinetic energy estimate, we will again use the property of the acceleration $G^N(x,v)$ and remark that $v$ in $G^N(x,v)$ is critical in the estimate because it serves as a damping term. We now choose $\{\varphi_{\eta}(x)\phi_{\eta}(v)\}$ to be a  smooth function which satisfies
$$
\varphi_{\eta}(x)=\begin{cases}
0, & |x| > \frac{1}{\eta},\\
1, & |x| <  \frac{1}{2\eta},
\end{cases}
\quad \hbox{and} \quad
\phi_{\eta}(v)=\begin{cases}
0, & |v| >\frac{1}{\eta},\\
1, & |v| < \frac{1}{2\eta},
\end{cases}
$$
and
$$
\Big|\nabla_z \Big(\varphi_{\eta}(x)\phi_{\eta}(v)\Big)\Big| \le \eta \Big|\varphi_{\eta}(x)\phi_{\eta}(v)\Big|.
$$
Since $\varphi_{\eta}(x)\phi_{\eta}(v)$ is monotone and converges to one for almost all $x$ and $v$ as $\eta$ goes to zero, we have
$$ \iint_{\R^2 \times \R^2}v^2f^N(t,x,v)\varphi_{\eta}(x)\phi_{\eta}(v)\,dxdv  \rightarrow \iint_{\R^2 \times \R^2}v^2f^N(t,x,v)\,dxdv, \quad \hbox{as} \, \eta \rightarrow 0.$$
The compact support of $f^N_0$ implies that $f^N(t,x,v)$ has compact support in $(x,v)$ for any fixed time $t$.
By the definition of weak solution for test functions $ v^2 \varphi_{\eta}(x)\phi_{\eta}(v) $, we have
%\begin{eqnarray*}
%&& \iint_{\R^2 \times \R^2} v^2f^N(t,x,v) \varphi_{\eta}(x)\phi_{\eta}(v)
%  \,dxdv \\
 % &=& \iint_{\R^2 \times \R^2} v^2f_0(x,v) \varphi_{\eta}(x)\phi_{\eta}(v)
 % \,dxdv + \int_0^t \iint_{\R^2 \times \R^2}vf^N(s,x,v) \nabla_x \left(v^2\varphi_{\eta}(x)\phi_{\eta}(v)\right) \,dxdvds  \nonumber \\
 % && + \int_0^t \iint_{\R^2 \times \R^2} (F^N(x,v)*f^N(s,x,v)) f^N(s,x,v)\nabla_v  \left(v^2\varphi_{\eta}(x)\phi_{\eta}(v)\right)\,dxdvds  \nonumber \\
  %  && + \int_0^t \iint_{\R^2 \times \R^2}G(x,v) f^N(s,x,v) \nabla_v  \left(v^2\varphi_{\eta}(x)\phi_{\eta}(v)\right) \,dxdvds.
  %\end{eqnarray*}
  \begin{eqnarray*}
&&\frac{d}{dt}\iint_{\R^2 \times \R^2} \frac{1}{2}v^2f^N(t,x,v) \varphi_{\eta}(x)\phi_{\eta}(v)
  \,dxdv \\
  &=& \frac{1}{2}\iint_{\R^2 \times \R^2}vf^N(t,x,v) \cdot\nabla_x \left(v^2\varphi_{\eta}(x)\phi_{\eta}(v)\right) \,dxdv  \nonumber \\
  && +  \frac{1}{2}\iint_{\R^2 \times \R^2} \left(F^N(x,v)*f^N(t,x,v)\right) f^N(s,x,v)\cdot\nabla_v  \left(v^2\varphi_{\eta}(x)\phi_{\eta}(v)\right)\,dxdv  \nonumber \\
    && +\frac{1}{2}\iint_{\R^2 \times \R^2}G^N(x,v) f^N(t,x,v) \cdot\nabla_v  \left(v^2\varphi_{\eta}(x)\phi_{\eta}(v)\right) \,dxdv
		\end{eqnarray*}
		\begin{eqnarray*}
     &=& \frac{1}{2}\iint_{\R^2 \times \R^2}v^2f^N(t,x,v)\phi_{\eta}(v)v\cdot \nabla_x \big(\varphi_{\eta}(x)\big) \,dxdv  \nonumber \\
     && + \iint_{\R^2 \times \R^2} v\left(F^N(x,v)*f^N(t,x,v)\right) f^N(t,x,v) \varphi_{\eta}(x)\phi_{\eta}(v)\,dxdv \\
     && +  \frac{1}{2}\iint_{\R^2 \times \R^2} v^2\left(F^N(x,v)*f^N(t,x,v)\right) f^N(s,x,v)\cdot\nabla_v  \left(\varphi_{\eta}(x)\phi_{\eta}(v)\right)\,dxdv  \nonumber \\
      && +\iint_{\R^2 \times \R^2}v\cdot G^N(x,v) f^N(t,x,v) \varphi_{\eta}(x)\phi_{\eta}(v) \,dxdv  \\
      && +\frac{1}{2}\iint_{\R^2 \times \R^2}v^2G^N(x,v) f^N(t,x,v)\cdot \nabla_v  \left(\varphi_{\eta}(x)\phi_{\eta}(v)\right) \,dxdv \\
      &= : & \sum_{j=1}^5 I_j.
       \end{eqnarray*}
Next, we estimate the expressions $I_j, j=1,\ldots,5$ individually.  It is easy to see
    \begin{eqnarray*}
    	 |I_1| &\le &  \frac{1}{2}\iint_{\R^2 \times \R^2} \Big|v^2f^N(t,x,v)\phi_{\eta}(v)v\cdot \nabla_x \big(\varphi_{\eta}(x)\big) \Big| \,dxdv \\
    	 & \le & \frac{1}{2} \eta \iint_{\R^2 \times \R^2} |v|^3f^N(t,x,v) |\phi_{\eta}(v) \varphi_{\eta}(x) | \,dxdv.
    \end{eqnarray*}
Due to the fact that $f_0^N$ is compactly supported, i.e., $f^N$ has also compact support for any finite time $t$, $I_1$ converges to zero as $\eta \rightarrow 0 $ for fixed $N$. The same argument holds for $I_3$ and $I_5$, i.e., $I_3$ and $I_5$ converge to zero as $\eta \rightarrow 0 $:
    \begin{eqnarray*}
  |I_3| &\le &  \frac{1}{2}\cdot C \eta \|F^N*f^N\|_{L^\infty}  \iint_{\R^2 \times \R^2} v^2 f^N(t,x,v)  \varphi_{\eta}(x)\phi_{\eta}(v)\,dxdv\\
  		&\le & \frac{1}{2}\cdot C \eta \|F^N\|_{L^\infty} \|f^N\|_{L^1}  \iint_{\R^2 \times \R^2} v^2 f^N(t,x,v)  \varphi_{\eta}(x)\phi_{\eta}(v)\,dxdv\\
   I_5 &\le & \frac{1}{2}\cdot  \eta \|j_{\frac1N}*g\|_{L^\infty}  \iint_{\R^2 \times \R^2}v^2 f^N(t,x,v) \varphi_{\eta}(x)\phi_{\eta}(v) \,dxdv \\
    && -\frac{1}{2} \eta \iint_{\R^2 \times \R^2} |v|^3f^N(t,x,v) \phi_{\eta}(v) \varphi_{\eta}(x) \,dxdv.
      \end{eqnarray*}
      However, for the other integral estimates, we need some extra calculations. Using the properties of the desired velocity and direction acceleration  $G^N(x,v)$, we arrive at
    \begin{eqnarray*}
    I_2 &\le & \|F^N*f^N\|_{L^\infty} \iint_{\R^2 \times \R^2} \left(\frac{1}{4\varepsilon}+\varepsilon v^2 \right)f^N(t,x,v) \varphi_{\eta}(x)\phi_{\eta}(v)\,dxdv\\
    &\le & \|F^N\|_{L^\infty}\|f^N\|_{L^1} \iint_{\R^2 \times \R^2} \left(\frac{1}{4\varepsilon}+\varepsilon v^2 \right)f^N(t,x,v) \varphi_{\eta}(x)\phi_{\eta}(v)\,dxdv\\
    I_4 &\le & \|j_{\frac1N}*g\|_{L^\infty} \iint_{\R^2 \times \R^2} \left(\frac{1}{4\varepsilon}+\varepsilon v^2 \right)f^N(t,x,v) \varphi_{\eta}(x)\phi_{\eta}(v)\,dxdv   \\ && -\iint_{\R^2 \times \R^2} v^2 f^N(t,x,v)  \varphi_{\eta}(x)\phi_{\eta}(v)\,dxdv
    \end{eqnarray*}

       Combining all the five terms, taking $\eta$ to zero in the inequality above and setting $\varepsilon $ small enough such that
 $$\varepsilon < \frac{1}{2 (\|F^N\|_{L^\infty}\|f^N\|_{L^1}+\|g\|_{L^\infty})},$$
 where the fact that $\displaystyle \|j_{\frac1N}*g\|_{L^\infty} \leq \|g\|_{L^\infty}$ has been used,
 we end up with
   \begin{eqnarray*}
\frac{d}{dt}\iint_{\R^2 \times \R^2} \frac{1}{2}v^2f^N(t,x,v) \,dxdv \le  C-\iint_{\R^2 \times \R^2} \frac{1}{2}v^2f^N(t,x,v) \,dxdv,
  \end{eqnarray*}
  where $C$ does not depend on $N$.
A direct computation shows that the kinetic energy is bounded uniformly in $t$ and $N$. The estimate for the second moment follows from
\begin{eqnarray*}
	\frac{d}{dt}\iint_{\R^2\times \R^2}|x|^2f^N(t,x,v)\,dxdv &=& \iint_{\R^2\times \R^2}|x|^2\partial_t f^N(t,x,v)\,dxdv \\
	&=& \iint_{\R^2\times \R^2}x\cdot vf^N(t,x,v)\,dxdv\\
	&\leq &  \iint_{\R^2\times \R^2}(|x|^2+|v|^2)f^N(t,x,v)\,dxdv\\
	&\leq &  \iint_{\R^2\times \R^2}|x|^2f^N(t,x,v)\,dxdv+C.
\end{eqnarray*}

%Moreover this solution satisfies
%\begin{eqnarray*}
 %&& \frac{d}{dt}\iint_{\R^4}vf^N(t,x,v) \,dxdv \\
  % &=& - \iint_{\R^4}v\nabla_v \cdot \left[(F^N*f^N)  f^N\right]\,dxdv-\iint_{\R^4}v\nabla_v \cdot (G f^N)
 %  \,dxdv \\
   % &=& \iint_{\R^4} (F^N*f^N)  f^N \,dxdv+\iint_{\R^4} Gf^N\,dxdv \\
    %& \le & \|F^N*f^N\|_{\infty}\mathcal{M}_0+\mathcal{M}_0-\iint_{\R^4}vf^N(t,x,v) \,dxdv.
   % \end{eqnarray*}
 %Therefore by the Gronwall's inequality we know
 %$$
 %\Big| \iint_{\R^4}vf^N(t,x,v) \,dxdv \Big| \le   \|F^N*f^N\|_{\infty}\mathcal{M}_0+\mathcal{M}_0.
 %$$
 %Furthermore the kinetic energy is also bounded because
%\begin{eqnarray*}
% && \frac{1}{2}\frac{d}{dt}\iint_{\R^2 \times \R^2}v^2f^N(t,x,v) \,dxdv \\
  %&=& -\frac{1}{2} \iint_{\R^2 \times \R^2}v^2\nabla_v \cdot \left[(F^N*f^N)  f^N\right]\,dxdv-\frac{1}{2} \iint_{\R^2 \times \R^2}v^2\nabla_v \cdot (G f^N)
%  \,dxdv \\
   %&=& \iint_{\R^2 \times \R^2} v[(F^N*f^N)  f^N] \,dxdv+\iint_{\R^2 \times \R^2} vGf^N\,dxdv \\
   %& \le & \|F^N*f^N\|_{\infty}\mathcal{M}_1+\mathcal{M}_1-\iint_{\R^2 \times \R^2}v^2f^N(t,x,v) \,dxdv.
   %\end{eqnarray*}
%By the same argument, we obtain
%$$
%\Big| \iint_{\R^2 \times \R^2}v^2f^N(t,x,v) \,dxdv \Big|  \le \Big( \|F^N*f^N\|_{\infty}+1 \Big)\Big(  \|F^N*f^N\|_{\infty}\mathcal{M}_0+\mathcal{M}_0
%\Big).
%$$
\end{proof}

%%%%%%%%%%%%%%%%%%%%%%%%%%%%%%%%%%%%%%%%%%%%%%%%%%%%%%%%%%%%%
\section{Compactness Arguments} \label{sec:compact}
%%%%%%%%%%%%%%%%%%%%%%%%%%%%%%%%%%%%%%%%%%%%%%%%%%%%%%%%%%%%%
In this section, we aim to achieve all the compactness arguments that are needed to pass the limit and to obtain the desired weak formulation of the non-cut-off kinetic equation, namely to prove the main result Theorem 2.1.

For given initial data $f_0$, let $f^N_0$ be a sequence of functions with compact support which are w.l.o.g. assumed to be in $B_N$, i.e., a ball of radius $N$ centered at the origin. Furthermore, $f^N_0$ satisfies
$$
\|f^N_0-f_0\|_{L^1(\R^2 \times \R^2)\cap L^\infty(\R^2 \times \R^2)}\to 0, \mbox{ as } N\to \infty.
$$
Let $f^N(t,x,v)$ be the solution obtained from Theorem \ref{thmexistN} with initial data $f^N_0(x,v)$.
Then, we know
$$
0 \leq f^N(t,x,v) \leq \Vert f_0 \Vert_{L^{\infty}(\R^{2}\times \R^{2})}e^{Ct}, \quad \hbox{for a.e.}  \,(x,v) \in \R^{2} \times \R^{2}, \,  t \geq 0,
$$
and for any fixed $T>0$, there exists a subsequence of $f^N$, still denoted by $f^N$ for simplicity, such that
$$
f^N  \overset {*}{\rightharpoonup } f \quad \hbox{in} \, L^{\infty}((0,T);L^{\infty}(\R^2 \times \R^2)).
$$
Due to the tightness in the variable $x$ and $v$ of the sequence $f^N$, implied from \eqref{kineticN} and \eqref{2momentN}, we conclude that $f \in L^1(\R^2 \times \R^2)$.
Moreover, we notice that the total mass is preserved, i.e.,
\begin{eqnarray*}
  \iint_{\R^2\times \R^2}f(t,x,v) \,dxdv=\iint_{\R^2\times \R^2}f^N_0(x,v) \,dxdv =: \mathcal{M}_0.
  \end{eqnarray*}
By the definition of weak* convergence for characteristic functions $\chi_{|x|+|v|\leq r} \in L^1(\R^2 \times \R^2)$, we have for each $a < b \in \R_+$
\begin{eqnarray*}
	&&\int_a^b \iint_{\R^2\times \R^2}\chi_{|x|+|v|\leq r}f(t, x, v)\, dxdvdt  \\
	&=&  \lim_{N \rightarrow \infty} \int_a^b\iint_{\R^2\times \R^2}\chi_{|x|+|v|\leq r}f^N(t, x,v)\,dxdvdt \\	
  &\le& \lim_{N \rightarrow \infty}\int_a^b\iint_{\R^2\times \R^2}f^N(t, x,v)\,dxdvdt =  \mathcal{M}_0(b-a).
  \end{eqnarray*}
  Letting $ r \rightarrow \infty$ and applying Fatou's lemma yields
  \begin{eqnarray*}
  &&\int_a^b \iint_{\R^2\times \R^2}f(t,x,v) \,dxdvdt  \\&\le& \varliminf_{r\rightarrow \infty}\int_a^b \iint_{\R^2\times \R^2}\chi_{|x|+|v|\leq r}f(t, x, v)\, dxdvdt \\
  &\le&  \lim_{N\rightarrow \infty}\int_a^b \iint_{\R^2\times \R^2}f^N(t,x,v) \,dxdvdt= \mathcal{M}_0(b-a).
  \end{eqnarray*}
%Secondly, we notice that the total mass is preserved, i.e.,
%\begin{eqnarray*}
%  \iint_{\R^2\times \R^2}f^N(t,x,v) \,dxdv=\iint_{\R^2\times \R^2}f^N_0(x,v) \,dxdv = \mathcal{M}_0,
%  \end{eqnarray*}
%from which by Fatou's lemma, we obtain for a.e. $t \ge 0$
%\begin{eqnarray*}
%  \iint_{\R^2\times \R^2}f(t,x,v) \,dxdv\le \varliminf_{N\rightarrow \infty}\iint_{\R^2\times \R^2}f^N(t,x,v) \,dxdv=\mathcal{M}_0,
%  \end{eqnarray*}
By a similar argument for test functions of type $ \chi_{|x|+|v|\leq r}|v|^{2}$, we can show that %one has for each $a < b \in \R_+$
$$
\int_{a}^{b}\iint_{\R^2\times \R^2}|v|^{2}f(t, x, v)\,dxdvdt \le  C(b-a)
$$
by using
$$
 \iint_{\R^{2}\times \R^{2}}\frac{1}{2}|v|^{2}f^N(t, x, v)\,dxdv \leq C(b-a), \quad \forall \, t \geq 0.
 $$
%$$
%\int_{a}^{b}\iint_{\R^2\times \R^2} \chi_{|x|+|v|\leq r}|v|^{2}f^N(t, x, v)\,dxdvdt  \rightarrow \int_{a}^{b}\iint_{\R^2\times \R^2}\chi_{|x|+|v|\leq r}|v|^{2}f(t,x, v)\,dxdvdt
%$$
%and since $f^N\geq 0$ a.e. on $\R_{+}\times \R^2\times \R^2$, we get
%\begin{eqnarray*}
%	\int_{a}^{b}\iint_{\R^2\times \R^2}\chi_{|x|+|v|\leq r}|v|^{2}f(t, x, v)\, dxdvdt  &=&  \varliminf_{N \rightarrow \infty}\int_{a}^{b}\iint_{\R^2\times \R^2}\chi_{|x|+|v|\leq r}|v|^{2}f^N(t, x,v)\,dxdvdt \\
%  &\le& \varliminf_{N \rightarrow \infty}\int_{a}^{b}\iint_{\R^2\times \R^2}|v|^{2}f^N(t, x,v)\,dxdvdt.
%\end{eqnarray*}
%Due to the fact that
%$$
% \iint_{\R^{2}\times \R^{2}}\frac{1}{2}|v|^{2}f^N(t, x, v)\,dxdv \leq C, \quad \forall \, t \geq 0,
% $$
%letting $ r \rightarrow \infty$ and applying Fatou's lemma show that
%$$ \int_{a}^{b}\iint_{\R^2\times \R^2}|v|^{2}f(t, x, v)\,dxdvdt \le  \varliminf_{N \rightarrow \infty}\int_{a}^{b}\iint_{\R^2\times \R^2}|v|^{2}f^N(t, x, v) \,dxdvdt \le C.
%$$
Since the above two inequalities hold for all $a<b \in \R_+$, they also hold for a.e. $t \in \R_{+}$.

Using all the estimates presented in Theorem \ref{thmexistN}, we are now ready to
 pass the limit in \eqref{vlasovN} to the desired weak formulation of the non-cut-off kinetic equation
\begin{eqnarray*}
%\label{VE}
\p_tf+v \cdot \nabla_xf+\nabla_v \cdot \left[(F*f)  f\right]+\nabla_v \cdot (G f)=0.
\end{eqnarray*}
However, we need to take special care on the non-linear term, i.e., the consideration of the function $F^N \ast f^N$.
%We remark again that the total force $F^N$ is considered on a bounded domain, i.e., the function has compact support within a ball of radius $\bar{R}:=\max\{2R, 2\widetilde{R}\}$, denoted by $B_{\bar{R}}$, where $R$ and $ \widetilde{R}$ are defined as
%$$
%\mathcal{H}_{2R}(|x|)=\begin{cases}
%0, & |x| >2R,\\
%1, & |x| < R,
%\end{cases}
%\quad \hbox{and} \quad
% \widetilde{\mathcal{H}}_{2\widetilde{R}}(|v|)=\begin{cases}
% 0, & |v| >2\widetilde{R},\\
% 1, & |v| < \widetilde{R}.
%\end{cases}
%$$
In the following, we use the notation $L^p(L^q)$ to denote $L^p([0,T]; L^q(\R^2 \times \R^2)), 1\le p,q \le \infty$.
It is obvious to see that
\begin{eqnarray*}
&&\|F^N \ast f^N \|_{	L^{\infty}(L^1)} \\
&=&  \Big\| \iint_{\R^2\times \R^2} \Big( \iint_{\R^2\times \R^2}F^N(x-y, v-w) f^N(t,y,w)\,dydw \Big) \,dxdv \Big\|_{L^{\infty}([0,T])} \\
&=&  \Big\| \iint_{\R^2\times \R^2} f^N(t,y,w)\Big( \iint_{\R^2\times \R^2}F^N(x-y, v-w) \,dxdv \Big) \,dydw \Big\|_{L^{\infty}([0,T])} \\
&\le & C\left( \|F\|_{L^{1}}, \mathcal{M}_0, \bar{R} \right)
\end{eqnarray*}
and
\begin{eqnarray*}
  \|F^N \ast f^N \|_{L^{\infty}(L^{\infty})} &=& \Big\| \iint_{\R^2\times \R^2} F^N(x-y, v-w) f^N(t,y,w)\,dydw
  \Big\|_{L^{\infty}(L^{\infty})}\\
  &\le & C \left( \|F\|_{L^{\infty}}, \mathcal{M}_0 \right).
\end{eqnarray*}
Since $\nabla_v F^N$ is bounded uniformly in $N$, we get
\begin{eqnarray*}
  &&\|\nabla_v \left(F^N \ast f^N \right) \|_{L^{\infty}(L^1)}  \\
  &=&   \Big\|\iint_{\R^2\times \R^2} \Big( \iint_{\R^2\times \R^2} \nabla_v F^N(x-y, v-w) f^N(t,y,w)\,dydw \Big) \,dxdv \Big\|_{L^{\infty}(\R_+)} \\
  &=&   \Big\|\iint_{\R^2\times \R^2}  f^N(t,y,w)\Big( \iint_{\R^2\times \R^2} \nabla_v F^N(x-y, v-w)\,dxdv \Big) \,dydw \Big\|_{L^{\infty}(\R_+)} \\
&\le & C \left( \|\nabla_v F\|_{L^{1}}, \mathcal{M}_0, \bar{R} \right)
\end{eqnarray*}
and
\begin{eqnarray*}
  \|\nabla_v \left( F^N \ast f^N \right) \|_{L^{\infty}(L^{\infty})} &=& \Big\|\iint_{\R^2\times \R^2} \nabla_v F^N(x-y, v-w) f^N(t,y,w)\,dydw
  \Big\|_{L^{\infty}(L^{\infty})} \\
  &\le & C \left( \|\nabla_v F\|_{L^{\infty}}, \mathcal{M}_0 \right).
\end{eqnarray*}
So far, we can conclude by interpolation that $F^N \ast f^N$ and $\nabla_v F^N \ast f^N$ are in $L^{\infty}(L^{2})$.
%If $\nabla_x F^N \ast f^N$ and $\p_t (F^N \ast f^N)$ are also in $L^{\infty}(L^2)$, the function $F^N \ast f^N$ then belongs to $L^{\infty}(W^{1,2})$, which allows us to use compact embedding theorem to further pass the limit in the non-linear term.
Furthermore, it holds
$$
  \|\nabla_x \left(F^N \ast f^N\right) \|_{L^{\infty}(L^{2})} \leq C \cdot \Big\|  \left(\chi_{\bar{R}} \cdot \frac{1}{|x|}\right)  \ast f^N \Big\|_{L^{\infty}(L^{2})} \le \|f^N\|_{L^{\infty}(L^{p})},\quad \forall \, p >1,
$$
where $\chi_{\bar{R}} \cdot \frac{1}{|x|} \in L^{r}, \forall \, 1<r<2,$ and Young's inequality have been used.
%$$\frac12+\frac1p+\frac1q=2, \quad 1<q<2, \quad q'>2.$$
%$$
  %\|\nabla_x F^N \ast f^N \|_{L^\infty} \leq C \cdot \Big\|\chi_{\bar{R}} \cdot\frac{1}{|x|}\ast f^N \Big\|_{L^\infty} \le \|f^N\|_{L^{\infty}}
%$$
Hence, we conclude that $F^N \ast f^N$ then belongs to $L^{\infty}(\R_+;W^{1,2}(\R^2\times \R^2))$.
Since
$$
\iint_{\R^2\times \R^2} \left(v f^N(t,x,v)\right)^2\,dxdv \le  \|f^N\|_{L^\infty}\|v^2f^N\|_{L^{\infty}(L^{1})} \le C(T),
$$
we can get for every $ \varphi \in C_0^{\infty}(\R^2\times \R^2) $ that
\begin{eqnarray}
  &&\Big\| \iint_{\R^2\times \R^2} v f^N(t,x,v) \nabla_x \varphi(x,v) \,dxdv\Big\|_{L^{\infty}(\R_+)} \nonumber \\
  &\le&  \|f^N\|^{\frac12}_{L^{\infty}(L^{\infty})}  \cdot \|v^2f^N\|^{\frac12}_{L^{\infty}(L^{1})} \cdot \|\nabla_x \varphi \|_{L^2} \nonumber \\
  &\le& C(T)\|\nabla_x \varphi \|_{L^2}.\label{vfN}
\end{eqnarray}
Moreover, we have
\begin{eqnarray} \label{GfN}
  && \Big\| \iint_{\R^2\times \R^2} G^N(x,v) f^N(t,x,v) \nabla_v \varphi(x,v) \,dxdv \Big\|_{L^{\infty}(\R_+)} \nonumber \\
  &\le&  \|j_{\frac1N}*g\|_{L^\infty}\cdot  \|f^N\|^{\frac12}_{L^{\infty}(L^{\infty})}  \cdot \|f^N\|^{\frac12}_{L^{\infty}(L^{1})}  \cdot \|\nabla_v \varphi \|_{L^2} \nonumber \\
  &&+\|f^N\|^{\frac12}_{L^{\infty}(L^{\infty})}  \cdot \|v^2f^N\|^{\frac12}_{L^{\infty}(L^{1})} \cdot \|\nabla_v \varphi \|_{L^2}\nonumber  \\
   &\le&  \|g\|_{L^\infty}\cdot  \|f^N\|^{\frac12}_{L^{\infty}(L^{\infty})}  \cdot \|f^N\|^{\frac12}_{L^{\infty}(L^{1})}  \cdot \|\nabla_v \varphi \|_{L^2} \nonumber \\
  &&+\|f^N\|^{\frac12}_{L^{\infty}(L^{\infty})}  \cdot \|v^2f^N\|^{\frac12}_{L^{\infty}(L^{1})} \cdot \|\nabla_v \varphi \|_{L^2}\nonumber  \\
  &\le& C(T)\|\nabla_v \varphi \|_{L^2}.
\end{eqnarray}
On the other hand, we know
\begin{eqnarray} \label{FNfNfN}
	 && \Big\|\iint_{\R^2 \times \R^2}\left(F^N*f^N\right)(t,x,v) \cdot f^N(t,x,v) \nabla_v \varphi(x,v)  \,dxdv  \Big\|_{L^{\infty}(\R_+)} \nonumber \\
	&\le&   \|F^N*f^N\|_{L^{\infty}(L^{\infty})}  \cdot \|f^N\|_{L^{\infty}(L^{2})}  \cdot \|\nabla_v \varphi \|_{L^2} \nonumber \\
  &\le& C\|\nabla_v \varphi \|_{L^2}.
\end{eqnarray}
%$$
%f^N  \overset {*}{\rightharpoonup } f \quad \hbox{in} \, L^{\infty}(\R_+;L^{\infty}(\R^2 \times \R^2))
%$$
%and
%$$
%F^N  \rightarrow F \quad \hbox{in} \, L^{2}(\R^2 \times \R^2).
%$$
%Consequently
%$$
%F^N \ast f^N \rightharpoonup  F \ast f \quad \hbox{in} \, L^{\infty}(\R_+;L^{2}(\R^2 \times \R^2))
%$$
%which tells us
%$$
%F^N*f^N \cdot f^N \in L^\infty(\R_+;L^2(\R^2 \times \R^2))
%$$
%For any given time $T>0$,
Combining \eqref{vfN}-\eqref{FNfNfN}, it holds for every $\varphi \in C_0^{\infty} \left(\R^2 \times \R^2\right) $ that
\begin{eqnarray*}
	&& \Big\| \iint_{\R^2 \times \R^2}  \p_t f^N(t,x,v)\varphi(x,v)\,dxdv \Big\|_{L^{\infty}(\R_+)} \\
%	&\le& \Big|\iint_{\R^2 \times \R^2}f^N(t,x,v) \nabla_x (v \varphi(x,v)) \,dxdv  \Big| \\
%	&& + \Big|\iint_{\R^2 \times \R^2}\left(F^N*f^N\right)(t,x,v) \cdot f^N(t,x,v) \nabla_v \varphi(x,v)  \,dxdv  \Big| \\
%	&& +\Big|\iint_{\R^2 \times \R^2} G(x,v)f^N(t,x,v) \nabla_v \varphi(x,v)  \,dxdv  \Big|  \\
	 &\le& \Big\|\iint_{\R^2 \times \R^2} v f^N(t,x,v) \nabla_x \varphi(x,v)  \,dxdv  \Big\|_{L^{\infty}(\R_+)}\\
	 && + \Big\|\iint_{\R^2 \times \R^2}\left(F^N*f^N\right)(t,x,v) \cdot f^N(t,x,v) \nabla_v \varphi(x,v)  \,dxdv  \Big\|_{L^{\infty}(\R_+)} \\
	&& +\Big\| \iint_{\R^2 \times \R^2} G^N(x,v)f^N(t,x,v) \nabla_v \varphi(x,v)  \,dxdv  \Big\|_{L^{\infty}(\R_+)}  \\
	 &\le& C \|\varphi\|_{W^{1,2}} ,
	 %&\le& CT^{\frac12}\|\varphi\|_{W^{1,2}} ,
\end{eqnarray*}
which implies
\begin{eqnarray*}
	&& \Big\| \iint_{\R^2 \times \R^2}  \p_t \Big((F^N \ast f^N)(t,x,v)\Big)\varphi(x,v)\,dxdv \Big\|_{L^{\infty}(\R_+)} \\
	 &=& \Big\| \iint_{\R^2 \times \R^2}  \p_t f^N(t,x,v) (F^N*\varphi)(x,v)\,dxdv \Big\|_{L^{\infty}(\R_+)} \\
	 &\le& C \|F^N*\varphi\|_{W^{1,2}}\\
	 &=& C  \Big\| \iint_{\R^2 \times \R^2}F^N(y,w)\varphi(x-y,v-w)\,dydw \Big\|_{W^{1,2}}\\
	 &\le& C\|F^N\|_{L^{\infty}} \|\varphi\|_{W^{1,2}}\\
	 &\le& C\|F\|_{L^{\infty}} \|\varphi\|_{W^{1,2}}
	 %&\le& CT^{\frac12}\|\varphi\|_{W^{1,2}} ,
\end{eqnarray*}
or, in other words,
$$
\| \p_t (F^N \ast f^N)\|_{L^{\infty}(W^{-1,2})} =\| F^N \ast \p_t f^N\|_{L^{\infty}(W^{-1,2})} \le C.
$$
We then get %$\forall \, t \in (0,T)$,
$\forall \, \varphi \in C^{\infty}_0(\R^2 \times \R^2)$
$$
F^N*f^N \in L^{\infty}( [0,T];W^{1,2}(\Omega)),  \quad \p_t (F^N \ast f^N) \in L^{\infty}( [0,T];W^{-1,2}(\Omega)),
$$
where $\Omega=\hbox{supp} \varphi$. According to Aubin-Lions compact embedding theorem, e.g. \cite{CJL14,Simon86}, there exists a subsequence and $h \in L^{\infty}( [0,T];L^{2}(\Omega)) $ such that
  $$
F^N \ast f^N  \rightarrow  h   \quad \hbox{in} \, L^{\infty}( [0,T];L^{2}(\Omega)).
$$
It is not difficult to check that $h=F \ast f $.
% $$
%F^N \ast f^N \cdot \varphi \rightarrow  F \ast f  \cdot \varphi  \quad \hbox{in} \, L^{\infty}(\R_+;L^{2}(\R^2 \times \R^2)).
%$$
Therefore we obtain the following estimates:
%\begin{eqnarray*}
% && \int_0^t \iint_{\R^2 \times \R^2} \left(F^N(x,v)*f^N(s,x,v)\right) f^N(s,x,v)\nabla_v \varphi(x,v)\,dxdvds \\
% &\rightarrow&  \int_0^t \iint_{\R^2 \times \R^2} \left(F(x,v)*f(s,x,v)\right) f(s,x,v)\nabla_v \varphi(x,v)\,dxdvds, \quad \hbox{as} \, N \rightarrow \infty.
%\end{eqnarray*}
%\begin{eqnarray*}
% && \lim_{N \rightarrow \infty} \Big|  \int_0^t\iint_{\R^2 \times \R^2} \Big(\left(F^N*f^N\right)(s,x,v) f^N(s,x,v)\nabla_v \varphi(x,v) \\
% && -\left(F*f\right)(s,x,v) f(s,x,v)\nabla_v \varphi(x,v)\Big)\,dxdvds \Big|\\
% &=&\lim_{N \rightarrow \infty} \Big| \int_0^t\iint_{\R^2 \times \R^2} \Big( \left(F^N*f^N\right)(s,x,v) f^N(s,x,v)\nabla_v \varphi(x,v)-\left(F*f\right)(s,x,v) f^N(s,x,v)\nabla_v \varphi(x,v) \\
% && + \left(F*f\right)(s,x,v) f^N(s,x,v)\nabla_v \varphi(x,v) - \left(F*f\right)(s,x,v) f(s,x,v)\nabla_v \varphi(x,v)\Big) \,dxdvds\Big|  \\
% &\le & \lim_{N \rightarrow \infty} \| F^N*f^N-F*f \|_{L^{\infty}(L^2)} \| f^N\|_{L^{\infty}(L^{\infty})}\| \nabla_v \varphi \|_{L^2} \\
% &=& 0,
 %&& +\Big|\iint_{\R^2 \times \R^2} \left(F*f\right)(t,x,v) f^N(t,x,v)\nabla_v \varphi(x,v) - \left(F*f\right)(t,x,v) f(t,x,v)\nabla_v \varphi(x,v)\Big) \,dxdv\Big|,
%\end{eqnarray*}
\begin{eqnarray*}
 &&  \Big|  \displaystyle\int_0^t\!\!\iint_{\R^2 \times \R^2} \!\!\!\Big(\left((F^N*f^N)f^N\right)(s,x,v)\nabla_v \varphi(x,v) - \big((F*f) f\big)(s,x,v)\nabla_v \varphi(x,v)\Big)\,dxdvds \Big|\\
 &=&\Big| \int_0^t\!\!\iint_{\Omega} \Big( \left((F^N*f^N)f^N\right)(s,x,v)\nabla_v \varphi(x,v)- \big((F*f) f^N\big)(s,x,v)\nabla_v \varphi(x,v) \\
 &&\hspace{0.5cm} +  \big((F*f) f^N\big)(s,x,v)\nabla_v \varphi(x,v) -  \big((F*f) f\big)(s,x,v)\nabla_v \varphi(x,v)\Big) \,dxdvds\Big|  \\
 &\le & \Big| \int_0^t\!\!\iint_{\Omega} \Big(  \left((F^N*f^N)f^N\right)(s,x,v)\nabla_v \varphi(x,v)- \big((F*f) f^N\big)(s,x,v)\nabla_v \varphi(x,v) \Big) \,dxdvds\Big|\\
&& \hspace{0.5cm}+\Big|\int_0^t\!\!\iint_{\Omega} \Big(\big((F*f) f^N\big)(s,x,v)\nabla_v \varphi(x,v) -  \big((F*f) f\big)(s,x,v)\nabla_v \varphi(x,v)\Big)
\,dxdvds\Big| \\
& =: &J_1+J_2.
\end{eqnarray*}
For the first term $J_1$, we have
$$ \lim_{N \rightarrow \infty} J_1 \le \lim_{N \rightarrow \infty} \| F^N*f^N-F*f \|_{L^{\infty}(L^2(\Omega))} \| f^N\|_{L^{\infty}(L^{\infty})}\| \nabla_v \varphi \|_{L^2}=0$$
while for the second term $J_2$ we use the fact that $f^N  \overset {*}{\rightharpoonup } f$ in $L^{\infty}(\R_+;L^{\infty}(\R^2 \times \R^2))$ for $F*f \cdot \nabla_v \varphi \in L^1(L^1)$, namely
$$ \lim_{N \rightarrow \infty} J_2 =0.$$
%The above procedure works for arbitrary $T$, so it can be extended to infinity by using the diagonal method in a straightforward way. %For the energy bound, we take $N \to \infty$ in \eqref{kineticN}, and one can easily see that
%$$
%\iint_{\R^{2}\times \R^{2}}\frac{1}{2}|v|^{2}f(t, x, v)\,dxdv \leq C, \quad \forall \, t \geq 0.
%$$

Finally, we have to examine the initial data. Since $f^N$ is the weak solution to the cut-off mean field equation \eqref{vlasovN}, it obviously satisfies
\begin{eqnarray*}
 &&\iint_{\R^2 \times \R^2} f^N(t,x,v) \varphi(x,v)
  \,dxdv = \iint_{\R^2 \times \R^2}f^N_0(x,v) \varphi(x,v)
  \,dxdv \nonumber \\
  && \hspace{16mm} +\int_0^t \iint_{\R^2 \times \R^2}vf^N(s,x,v) \cdot\nabla_x \varphi(x,v)
  \,dxdvds  \nonumber \\
  && \hspace{16mm}+\int_0^t \iint_{\R^2 \times \R^2} \left(F^N(x,v)*f^N(s,x,v)\right) f^N(s,x,v)\cdot\nabla_v \varphi(x,v)\,dxdvds  \nonumber \\
    &&\hspace{16mm}+ \int_0^t \iint_{\R^2 \times \R^2}G^N(x,v) f^N(s,x,v)\cdot \nabla_v \varphi(x,v) \,dxdvds
  \end{eqnarray*}
for any test function $ \varphi(x,v) \in C^{\infty}_0(\R^2 \times \R^2)$.
We recall	
$$
\|f^N_0-f_0\|_{L^1(\R^2 \times \R^2)\cap L^\infty(\R^2 \times \R^2)}\to 0, \mbox{ as } N\to \infty,
$$
and that terms on the right (second till last) hand side are uniformly continuous in time $t$.
Then, taking limit $t \rightarrow 0^+$ on both sides of the above equation verifies the initial data.

\section{Summary}
This paper deals with the core problem, which is to show existence of the $L^{\infty}( (0,\infty);$ $L^{\infty}(\R^2\times \R^2))$-solution to the mean field kinetic equation for interacting particle systems with non-Lipschitz force. Our main results, Theorem \ref{thm:1} and Theorem \ref{thmexistN}, state that there exists a weak solution to the mean field equation (or approximate equation with cut-off) to the interaction flow model. The solution is proven to satisfy the mass conservation and energy bounds, respectively. In particular, this paper addresses technical difficulties caused by the non-Lipschitz continuous interaction force and self-generating force.

\section*{Acknowledgments}
This work was financially supported by the DAAD project “DAAD-PPP VR China”
(Project-ID: 57215936) and the Deutsche Forschungsgemeinschaft (DFG) Grant CH 955/4-1.

\end{document}